\documentclass{amsart}
\pdfoutput 1

\usepackage{mathtools}
\usepackage[all,hyperref,numberbysection]{bi-discrete}
\usepackage[backend=biber,maxbibnames=5,maxalphanames=5,style=alphabetic,bibencoding=utf8,giveninits,url=false,isbn=false]{biblatex}
\AtBeginBibliography{\small}

\newcommand{\px}{\ensuremath{{p(\cdot)}}}
\newcommand{\pdx}{\ensuremath{{p'(\cdot)}}}

\newcommand{\rhoKR}{\rho^{\scriptscriptstyle\text{KR}}}
\newcommand{\rhoER}{\rho^{\scriptscriptstyle\text{ER}}}

\begin{document}
	
\title[Maximal operator on variable exponent spaces]{Maximal operator on variable exponent spaces}%

\author{Daviti Adamadze}
\address{Daviti Adamadze \\
  Fakult\"{a}t  f\"{u}r Mathematik,
  University of Bielefeld\\
  Universit\"{a}tsstrasse 25, 33615 Bielefeld, Germany}
\email{daviti.adamadze@uni-bielefeld.de}

\author{Lars Diening}
\address{Lars Diening \\
  Fakult\"{a}t  f\"{u}r Mathematik,
  University of Bielefeld\\
  Universit\"{a}tsstrasse 25, 33615 Bielefeld, Germany}
\email{lars.diening@uni-bielefeld.de}

\author{Tengiz Kopaliani}
\address{Tengiz Kopaliani \\
  Faculty of Exact and Natural Sciences\\
  Javakhishvili Tbilisi State
  University\\
  13, University St., Tbilisi, 0143, Georgia}
\email{tengiz.kopaliani@tsu.ge}
\thanks{}%
\thanks{The first author gratefully acknowledges financial support by the Deutsche Forschungsgemeinschaft (DFG, German Research Foundation) - IRTG 2235 - Project number 282638148. The third author was supported by Shota Rustaveli National Science Foundation of Georgia FR-21-12353.}%

\begin{abstract}
  We explore the boundedness of the Hardy-Littlewood maximal operator~$M$ on variable exponent spaces. Our findings demonstrate that the Muckenhoupt condition, in conjunction with Nekvinda's decay condition, implies the boundedness of~$M$ even for unbounded exponents. This extends the results of Lerner, Cruz-Uribe and Fiorenza for bounded exponents. We also introduce a novel argument that allows approximate unbounded exponents by bounded ones while preserving the Muckenhoupt and Nekvinda conditions.
\end{abstract}

\maketitle

{\centerline{\small \emph{Dedicated to the loving memory of our friend Ale{\v s} Nekvinda}.}\vspace{0mm}}



\section{Introduction}
\label{sec:introduction}

Variable exponent spaces are an extension of the classical Lebesgue spaces, where the exponent is not a constant but depends on the spatial variable. These function spaces appear naturally in the study of electrorheological fluids, where the viscosity changes significantly by the presence of an electric field, see \cite{RajRuz1996, Ruzicka2000book,AcerbiMingione2002}. Given a variable exponent $p\,:\, \RRn \to [1,\infty]$ a measurable function $f\,:\, \RRn \to \RR$ is in the variable exponent Lebesgue space~$L^{p(\cdot)}(\RRn)$ if
\begin{align}
  \label{eq:modular-intro}
  \rho_{\px}(f) &\coloneqq \int_{\RRn} \abs{f(x)}^{p(x)}\,dx < \infty
\end{align}
with the convention that $t^\infty \coloneqq 0$ for $t \in [0,1]$ and  $t^\infty \coloneqq \infty$ for $t >1$.
The norm on~$L^{p(\cdot)}(\RRn)$ is then defined by 
\begin{align}
  \label{eq:norm-intro}
  \norm{f}_{p(\cdot)}\coloneqq\inf \left\lbrace \lambda > 0: 
    \rho_{p(\cdot)} (f/ \lambda) \leq 1 \right\rbrace.
\end{align}
If $p$ is constant, then we just obtain the classical Lebesgue spaces.

Since the introduction of the model on electrorheological fluids, there has been a tremendous increase in research on variable exponent spaces. Many results regarding variable exponent Lebesgue spaces and variable exponent Sobolev spaces are by now classical and can be found in the books \cite{DieHarHasRuz11,CruFio13}.

The Hardy-Littlewood maximal operator~$M$ plays a fundamental role in the investigation of those spaces and their related PDEs. It is defined by
\[
	Mf(x) =\sup_{x\in Q}\frac{1}{|Q|}\int_{Q}|f(y)|dy,
\]
where the supremum is taken over all cubes $Q\subset\mathbb{R}^{n}$ containing the point $x$ and $|Q|$ denotes the Lebesgue measure of~$Q$. Similar versions of~$M$ with balls or centered balls can also be used. 

Whenever~$M$ is bounded on~$L^{p(\cdot)}(\RRn)$, then many results like sharp embeddings and boundedness of singular integral operators follow by standard tools from harmonic analysis. Hence, it is very important to have an effective criterion on the variable exponent in order to guarantee the boundedness of~$M$. 

In a series of papers starting with Diening~\cite{Die04} and Cruz-Uribe, Fiorenza and Neugebauer~\cite{CruFioNeu03} it was shown that if~$p$ satisfies a local $\log$-H\"older condition together with a $\log$-decay condition and if~$p^-\coloneqq \essinf p>1$, then $M$ is bounded on~$L^{p(\cdot)}(\RRn)$, see \cite[Theorem~4.3.8]{DieHarHasRuz11}. The conditions on~$p$ are
\begin{align*}
  \bigabs{\tfrac 1{p(x)} - \tfrac 1{p(y)}} &\leq \tfrac{c}{\log(e+1/\abs{x-y})}
  &
  \bigabs{\tfrac 1{p(x)} - \tfrac 1{p_\infty}} &\leq \tfrac{c}{\log(e+1/\abs{x})}
\end{align*}
for some~$p_\infty \in [1,\infty]$ and all $x,y\in \RRn$.

The decay condition was relaxed by Nekvinda~\cite{Nek04} to an integral condition, i.e. $p \in \mathcal{N}$ if there exists 
$p_{\infty} \in [1,\infty]$ such that
\begin{align}
  \label{eq:def-N}
  1 \in L^{s(\cdot)}(\mathbb{R}^{n}) \qquad \text{with} \qquad
  \tfrac{1}{s(x)} \coloneqq \bigabs{\tfrac{1}{p(x)} - \tfrac{1}{p_{\infty}}}.
\end{align}
We define $[p]_{\mathcal{N} } \coloneqq \norm{1}_{s(\cdot)}$. Some of the first proofs were limited to the case~$p^+ \coloneqq \esssup p<\infty$. The proof for $p^+=\infty$ goes back to~\cite[Theorem 6.3]{DieHarHasMizShi09}, where it is also shown that~$p^->1$ is necessary.

Although the criterion using local $\log$-H\"older continuity is easy to check it is unfortunately not necessary. It is the optimal criterion in terms of a modulus of continuity, see \cite{PickRuzicka01} and \cite[after Remark 4.7.6]{DieHarHasRuz11}. However, there are even non-continuous variable exponents~$p$ with no limit at~$\infty$ such that $M$ is still bounded~\cite{Lerner05} and \cite[Example~5.1.8]{DieHarHasRuz11}.

Since the $\log$-H\"older conditions were not optimal, it was natural to investigate integral conditions that are similar to the conditions on Muckenhoupt weights, \cite{Die05}. It is well known for $1<p< \infty$ that $M$ is bounded on the weighted Lebesgue space $L^p(w\,dx)$ with norm $(\int_{\RRn} \abs{f(x)}^{p} w(x)\,dx)^{\frac 1p}$ if and only if~ $w\,:\,\RRn \to (0,\infty)$ is a~$A_p$-Muckenhoupt weight. The standard definition of a Muckenhoupt weight is not very useful in our context, since our~$p$ is variable. However, the $A_p$-Muckenhoupt condition is equivalent to the uniform boundedness of the family operators~$T_Q f \coloneqq \indicator_Q \dashint_Q \abs{f(x)}\,dx$ on $L^p(w\,dx)$, where~$Q$ is a cube and $\indicator_Q$ its indicator function.

Another simple characterization of Muckenhoupt weight uses the conjugate spaces. The conjugate space of $L^p(w \,dx)$ is $L^{p'}(w^{\frac{1}{1-p}} \,dx)$, where $p'$ is the conjugate exponent of~$p$, i.e. $1 = \frac 1p + \frac 1{p'}$.  Then $w$ is an $A_p$-Muckenhoupt weight if and only if
\begin{align*}
  \sup_Q \tfrac{1}{|Q|} \|\indicator_Q\|_{L^p(w,\,dx)} \|\indicator_Q\|_{L^{p'}(w^{\frac{1}{1-p}} \,dx)} < \infty,
\end{align*}
where $\indicator_Q$ is the indicator function of~$Q$. This condition can be generalized to the language of variable exponents: We say that~$p \in \mathcal{A}$ if and only if
\begin{align} \label{eq:Ap-intro}
  [p]_{\mathcal{A}} = \sup_Q \tfrac{1}{|Q|} \|\indicator_Q\|_{p(\cdot)} \|\indicator_Q\|_{p'(\cdot)} < \infty,
\end{align}
This condition is equivalent to the uniform boundedness of the family~$T_Q$ on~$L^{p(\cdot)}(\RRn)$ as shown in \cite[Theorem 4.5.7]{DieHarHasRuz11}. Moreover, the condition is obviously stable under duality, i.e $[p']_{\mathcal{A}} = [p]_{\mathcal{A}}$. 

However, unlike the case of weighted Lebesgue spaces, the Muckenhoupt condition~$p \in \mathcal{A}$ is not sufficient for the boundedness of~$M$, see \cite{Kop08} and \cite[Theorem~5.3.4]{DieHarHasRuz11}. This problem was overcome in~\cite{Die05} by using averaging operator over disjoint families of cubes. In particular, it was shown in~\cite{Die05} that the boundedness of~$M$ (for $1<p^-\leq p^+<\infty$) is equivalent to the uniform boundedness of the family of operators~$T_{\mathcal{Q}}$ with $T_\mathcal{Q} f\coloneqq \sum_{Q \in \mathcal{Q}} T_Q f$, where~$\mathcal{Q}$ runs over the families of pairwise disjoint cubes. An important consequence of the characterization in~\cite{Die05} is that $M$ is bounded on $L^{p(\cdot)}(\RRn)$ if and only if it is bounded on~$L^{p'(\cdot)}(\RRn)$ provided that $1<p^-\leq p^+<\infty$. Unfortunately, the condition in~\cite{Die05} using the family of~$T_{\mathcal{Q}}$ is difficult to check. Another equivalent condition was introduced recently in~\cite{Lerner2023arXiv} for $1<p^-\leq p^+ < \infty$, but it is also difficult to verify.

Kopaliani~\cite{Kop07} demonstrated that if~$p$ is constant outside a large ball and $1<p^-\leq p^+<\infty$, then~$p\in\mathcal{A}$ is already sufficient for the boundedness of~$M$ on $L^{p(\cdot)}(\RRn)$. An alternative proof of this fact has been provided by Lerner~\cite[Remark~3.4]{Ler10}. In this article Lerner also provides a local version estimating $\norm{\indicator_E Mf}_{\px}$ for sets~$E$ of bounded measure, which is solely based on $p \in \mathcal{A}$ and $1<p^-\leq p^+ < \infty$. This technique was refined by Cruz-Uribe and Fiorenza in~\cite[Theorem~4.52]{CruFio13} to show boundedness of~$M$ on $\RRn$ with the additional assumption that $p \in \mathcal{N}$ still using $1<p^-\leq p^+<\infty$. The purpose of this paper is to extend the result to the case~$p^+\leq \infty$.
Our main result is as follows:
\begin{theorem}
  \label{thm:main}
  Let $p\in \mathcal{A}\cap\mathcal{N}$ with $p^->1$. Then there exists a constant $c>0$ depending on $(p')^+,\,[p]_{\mathcal{A}},\,[p]_{\mathcal{N}}$ and $n$ such that for any $f\in L^{p(\cdot)}(\mathbb{R}^{n})$
  \begin{align}
    \label{eq:maininequality}
    \norm{Mf}_{p(\cdot)}\leq c\,\norm{f}_{p(\cdot)}.
  \end{align}
\end{theorem}

\section{Variable exponent spaces}
\label{sec:preliminaries}

Let us repeat some basic facts on Lebesgue spaces with variable exponents. We mainly use the notation of~\cite{DieHarHasRuz11}, but another good reference is~\cite{CruFio13}.

A variable exponent $p$ on~$\RRn$ is a measurable function $p\,:\, \RRn \to [1,\infty]$. The set of variable exponents on~$\RRn$ is denoted by $\mathcal{P}(\RRn)$. We define
\begin{align}
  \label{eq:pplusminus}
  p^- &\coloneqq \essinf_{x \in \RRn} p(x) \qquad \text{and} \qquad
  p^+ \coloneqq \esssup_{x \in \RRn} p(x).
\end{align}
By $p'$ we denote the conjugate exponent, i.e. $\frac{1}{p} + \frac{1}{p'}=1$.  Let $L^0(\RRn)$ denote the set of measurable functions on~$\RRn$. Then, for $f \in L^0(\RRn)$ we define the semi-modular
\begin{align}
  \label{eq:modular}
  \rho_{\px}(f) &\coloneqq \int_{\RRn}  \abs{f(x)}^{p(x)}\,dx
\end{align}
with the convention that $t^\infty \coloneqq 0$ for $t \in [0,1]$ and $t^\infty \coloneqq \infty$ for $t >1$.  The variable Lebesgue space $L^{p(\cdot)}(\RRn)$ is the set of $f \in L^0(\RRn)$ such that $\norm{f}_\px < \infty$, where the norm is defined by
\begin{align}
  \label{eq:norm}
  \norm{f}_{p(\cdot)}\coloneqq\inf \left\lbrace \lambda > 0: 
    \rho_{p(\cdot)} (f/ \lambda) \leq 1 \right\rbrace.
\end{align}
With this norm $L^{p(\cdot)}(\RRn)$ is a Banach space, see \cite[Theorem~3.2.7]{DieHarHasRuz11}. If $p$ is constant, then we recover the classical Lebesgue spaces.

Note that $\rho_\px(f) \leq 1$ is equivalent to $\norm{f}_{\px} \leq 1$.  Moreover, if $p^+<\infty$, then $\rho_\px(f) = 1$ is equivalent to $\norm{f}_{\px} = 1$. Both properties are called \emph{unit ball property}, see \cite[Lemma~2.1.14]{DieHarHasRuz11}. We will also use the following simple trick:
\begin{align}
  \label{eq:trick}
  \text{If $\rho_{\px}(f) \leq a b$ with $a \geq 1$ and $b \in (0,1]$, then $\norm{f}_{\px} \leq a b^{\frac{1}{p^+}}$.}
\end{align}
\begin{remark}
  \label{rem:modulars}
  There are other semi-modulars and modulars that can be used to define $L^{\px}(\RRn)$ with equivalent norms. The most prominent are
  \begin{subequations}
    \label{eq:alternative-modulars}
    \begin{align}
      \label{eq:alternative-modulars-tile}
      \tilde{\rho}_{\px}(f) &\coloneqq \int_{\RRn} \frac{1}{p(x)} \abs{f(x)}^{p(x)}\,dx,
      \\
      \rhoER_{\px}(f) &\coloneqq \max \biggset{\int_{\set{p \neq \infty}}  \abs{f(x)}^{p(x)}\,dx, \norm{f\,\indicator_{\set{p=\infty}}}_\infty},
      \\
      \rhoKR_{\px}(f) &\coloneqq \int_{\set{p \neq \infty}}  \abs{f(x)}^{p(x)}\,dx + \norm{f\,\indicator_{\set{p=\infty}}}_\infty.
    \end{align}
  \end{subequations}
  Note that $\tilde{\rho}_\px$ and $\rho_\px$ are semi-modulars and $\rhoER_{\px}$ and $\rhoKR_{\px}$ are modulars, e.g. $\rhoER_{\px}(f)=0$ implies~$f=0$.  If $p^+<\infty$ or $\set{p=\infty}$ has measure zero, then $\rho_{\px}(f) = \rhoER_\px(f) = \rhoKR_\px(f)$.  We have the following relations
  \begin{subequations}
    \begin{alignat}{4}
      \tilde{\rho}_\px(f) &\leq \rho_\px(f) &&\leq \rhoER_\px(f) &&\leq \rhoKR_\px(f)
      \\
      \tilde{\rho}_\px(f) &\leq \rho_\px(f) &&\leq \tilde{\rho}_{\px}(2f),
    \end{alignat}
  \end{subequations}
  The induced norms satisfy the relation
  \begin{align}
    \label{eq:norm-equivalence}
    \norm{f}_{\tilde{\rho}_\px} &\leq \norm{f}_\px = \norm{f}_{\rhoER_\px} \leq \norm{f}_{\rhoKR_\px} \leq 2 \norm{f}_{\tilde{\rho}_\px}.
  \end{align}
  The equality $\norm{f}_\px = \norm{f}_{\rhoER_\px}$ has been shown in~\cite{KarSha22}. All other relations can be found in \cite[Lemma 3.1.6 and Remark~3.2.3]{DieHarHasRuz11}. The modular~$\rhoER_{\px}$ was introduced in~\cite{EdmundsRakosnik2000}. The modular $\rhoKR$ appeared in \cite{KovRak91} and is mainly used in the book \cite{CruFio13}. The semi-modulars~$\rho_\px$ and $\tilde{\rho}_\px$ are studied in~\cite{DieHarHasRuz11}.
\end{remark}
For $p \in \mathcal{P}(\RRn)$ we define its dual exponent~$p' \in \mathcal{P}(\RRn)$ by $\frac{1}{p(x)} + \frac{1}{p'(x)} = 1$.  Then we have the generalized H\"older's inequality, see \cite[Lemma~3.2.20]{DieHarHasRuz11}: If $p,q,s \in \mathcal{P}(\RRn)$ with $\frac{1}{s(x)} = \frac{1}{p(x)} + \frac{1}{q(x)}$, then we have H\"older's inequality
\begin{align}
  \label{eq:genhoelder}
  \norm{f g}_{s(\cdot)} &\leq 2\,\norm{f}_{\px} \norm{f}_{q(\cdot)}.
\end{align}
If $s \geq 1$ and $p \in \mathcal{P}(\RRn)$, then
\begin{align}\label{eq:multiplepower}
  \bignorm{ \abs{f}^s }_{\px} &=
  \bignorm{f }_{s\px}^s.
\end{align}
If $p_0,p_1 \in \mathcal{P}(\RRn)$ and $\frac{1}{p_\theta(x)} = \frac{1-\theta}{p_0(x)} + \frac{\theta}{p_1(x)}$ for some $\theta \in [0,1]$, then we have the following interpolation estimate
\begin{align}
  \label{eq:interpolation}
  \norm{f}_{p_\theta(\cdot)} \leq 2 \norm{f}^{1-\theta}_{p_0(\cdot)} \norm{f}^\theta_{p_1(\cdot)}.
\end{align}
which follows by~\eqref{eq:genhoelder} and~\eqref{eq:multiplepower}.

If $p^+<\infty$, then the  dual space of $L^\px(\RRn)$ can be identified with $L^\pdx(\RRn)$, see \cite[Theorem~3.4.6]{DieHarHasRuz11}. In the general case $p^+ \leq \infty$ the space $L^\pdx(\RRn)$  is still the associate Banach function space of~$L^\px(\RRn)$, see \cite[Theorem~3.2.13]{DieHarHasRuz11} and we have the following characterization of the norm, see \cite[Corollary~3.2.14]{DieHarHasRuz11}.
\begin{lemma}[Norm conjugate formula]
  \label{lem:norm-formula}
  Let $p \in \mathcal{P}(\RRn)$. Then
  \begin{align*}
    \tfrac 12 \norm{f}_\px
    \leq \sup_{\substack{g \in L^\pdx(\RRn) \\ \norm{g}_{\pdx} \leq 1}} \int_\RRn \abs{f}\,\abs{g}\,dx
    \leq 2 \norm{f}_\px.
  \end{align*}
  The supremum is unchanged if we only use~$g$, which are bounded and have compact support.
\end{lemma}

Let $r,s \in \mathcal{P}(\RRn)$. Since $L^{r(\cdot)}(\RRn) \embedding L^1_{\loc}(\RRn)$ and $L^{s(\cdot)}(\RRn) \embedding L^1_{\loc}(\RRn)$ we can consider the usual intersection $L^{r(\cdot)}(\RRn) \cap L^{s(\cdot)}(\RRn)$ and sum $L^{r(\cdot)}(\RRn) + L^{s(\cdot)}(\RRn)$ with the norms
\begin{align*}
  \|f\|_{L^{r(\cdot)} \cap L^{s(\cdot)}} &= \max\{\|f\|_{r(\cdot)}, \|f\|_{s(\cdot)}\}, \\
  \|f\|_{L^{r(\cdot)} + L^{s(\cdot)}} &= \inf_{\set{f = g + h,\, g \in L^{r(\cdot)}, h \in L^{s(\cdot)}}} (\|g\|_{r(\cdot)} + \|h\|_{s(\cdot)}).
\end{align*}
Then we have the following embedding result proved in \cite[Theorem~3.3.11]{DieHarHasRuz11}
\begin{lemma}
  \label{lem:embedding}
  Let $p,q,r \in \mathcal{P}(\RRn)$ with $r \leq p \leq s$. Then
  \begin{align*}
    L^{r(\cdot)}(\RRn) \cap
    L^{s(\cdot)}(\RRn) \embedding
    L^{p(\cdot)}(\RRn)
    \embedding
    L^{r(\cdot)}(\RRn) +
    L^{s(\cdot)}(\RRn).
  \end{align*}
  The embedding constants are at most~$2$. The set~$\RRn$ can be replaced by subsets.
\end{lemma}


\section{Limits of functions and exponents}
\label{sec:limits-funct-expon}
  
In this section we study limits of functions and variable exponents in the setting of possibly unbounded variable exponents. The properties depend strongly on the chosen semi-modular.
Let us begin with limits of functions. It has been shown in \cite[Corollary~3.4.10]{DieHarHasRuz11} that $L^\px(\RRn) \cap L^\infty(\RRn)$ is dense in~$L^\px(\RRn)$ if~$p^+<\infty$. However, \cite{Kalyabin2007} showed that in general the condition~$p^+<\infty$ cannot be dropped. This makes approximations in~$L^\px(\RRn)$ difficult. Therefore, we will rely on the following convergence result, see \cite[Lemma~3.2.8, Theorem~2.3.17]{DieHarHasRuz11}.
\begin{lemma}[Fatou property]
  \label{lem:fatou}
  Let $f_k,f \in L^0(\RRn)$ and $\abs{f_k} \nearrow \abs{f}$ almost everywhere. Then
  \begin{align*}
    \rho_\px(f) &= \lim_{k \to \infty} \rho_{\px}(f_k) \qquad \text{and} \qquad
    \norm{f}_\px = \lim_{k \to \infty} \norm{f_k}_\px
    \\
    \tilde{\rho}_\px(f) &= \lim_{k \to \infty} \tilde{\rho}_{\px}(f_k) \qquad \text{and} \qquad
    \norm{f}_{\tilde{\rho}_\px} = \lim_{k \to \infty} \norm{f_k}_{\tilde{\rho}_\px}.
  \end{align*}
\end{lemma}
The approximation of the variable exponents~$p$ by a sequence $p_k$ is more subtle and depends strongly on the chosen semi-modular. The best semi-modular in this context is~$\tilde{\rho}_\px$. The following lower semi-continuity has been proved in \cite[Corollary~3.5.4]{DieHarHasRuz11}.
\begin{lemma}[Lower semi-continuity]
  \label{lem:lsc}
  Let $p_k,p \in \mathcal{P}(\RRn)$ with $p_k \to p$ almost everywhere. Then
  \begin{align*}
    \rho_{p(\cdot)}(f) &\leq \liminf_{k \to \infty} \rho_{p_k(\cdot)}(f) \qquad \text{and} \qquad
    \norm{f}_{\px} \leq \liminf_{k \to \infty} \norm{f}_{p_k(\cdot)},
    \\
    \tilde\rho_{p(\cdot)}(f) &\leq \liminf_{k \to \infty} \tilde\rho_{p_k(\cdot)}(f) \qquad \text{and} \qquad
    \norm{f}_{\tilde{\rho}_{p(\cdot)}} \leq \liminf_{k \to \infty} \norm{f}_{\tilde{\rho}_{p_k(\cdot)}}.
  \end{align*}
\end{lemma}
In some cases it is possible to obtain equality in the convergence of the modular. The following has been proved in~\cite[Lemma~3.5.5]{DieHarHasRuz11}.
\begin{lemma}[Approximation of semi-modular]
  \label{lem:approxpxmodular}
  Let $r,s,p_k,p \in \mathcal{P}(\RRn)$ such that $r \leq p_k \leq s$ and $p_k \to p$ almost everywhere. Then for each $f \in L^0(\RRn)$ with $\tilde\rho_{r(\cdot)}(f),\tilde\rho_{s(\cdot)}(f)<\infty$, there holds
  \begin{align*}
    \lim_{k \to \infty} \tilde\rho_{p_k(\cdot)}(f) = \tilde\rho_{p(\cdot)}(f).
  \end{align*}
\end{lemma}
\begin{remark}
  \label{rem:approximation of semi-modular}
  The assumptions $\tilde\rho_{r(\cdot)}(f), \tilde\rho_{s(\cdot)}(f)<\infty$ in Lemma~\ref{lem:approxpxmodular} cannot be removed.
  For the necessity of~$\tilde\rho_{s(\cdot)}(f)<\infty$, use $p_k = \infty \cdot\indicator_{(\frac 1{k+1},\frac 1{k})} + \indicator_{\RR \setminus(\frac 1{k+1},\frac 1{k})}$ and $f=2\cdot \indicator_{(0,1)}$. Then $p = \lim p_k = 1$, $r=\inf_k p_k=1$ and $s= \sup_k p_k= \infty \cdot \indicator_{(0,1)}$. Moreover, $\tilde\rho_{p_k(\cdot)}(f)=\infty$, $\tilde\rho_{r(\cdot)}(f)=2$, $\tilde\rho_{s(\cdot)}(f)=\infty$ but $\tilde\rho_{p(\cdot)}(f) = 2$.

  For the necessity of~$\tilde\rho_{r(\cdot)}(f)<\infty$, use $p_k \coloneqq 2\cdot\indicator_{(-\infty,k)} + \indicator_{[k,\infty)}$ and $f(x) = \indicator_{x\geq 1} \frac 1x$. Then $p = \lim_k p_k = 2$, $r = \inf_k p_k = 1$ and $s=\sup_{p_k} = 2$. Moreover, $\tilde\rho_{p_k}(f) =\infty$ and $\tilde\rho_{r(\cdot)}(f)=\infty$, but $\tilde\rho_{p(\cdot)}(f) =\tilde\rho_{s(\cdot)}(f) = \int_1^\infty \frac{1}{2} \frac{1}{x^2}\,dx = \frac 12$.
\end{remark}
\begin{remark}
  Lemma~\ref{lem:approxpxmodular} depends on the good properties of our semi-modular~$\tilde\rho_{p(\cdot)}$.  The semi-modular $\rho_\px$, $\rhoER_\px$ and $\rhoKR$ only have this property if the set $\set{s=\infty}$ has measure zero, see \cite[Lemma~3.5.5]{DieHarHasRuz11}, so this is an important case for us. But we want to use Lemma~\ref{lem:approxpxmodular} exactly to treat unbounded exponents. The counterexamples for $\rho_\px$, $\rhoER_\px$ and $\rhoKR$ are as follows:
  \begin{enumerate}
  \item Let $p_k = k$, $r=1$, $p=s= \infty$ and $f = \indicator_{(0,1)}$. Then $p_k\to \infty$, $\rho_1(f)=1$, $\rho_{p_k(\cdot)}(f)=1$,  and $\rho_\infty(f)=0$.
  \item Let $p_k =k$, $r=1$, $p=s=\infty$ and $f= 2 \cdot \indicator_{(0,1)}$. Then $p_k \to \infty$, $\rhoER_1(f) = \rhoKR_1(f)=2$, $\rhoER_\infty(f) = \rhoKR_\infty(f)=2$ and $\rhoER_{p_k(\cdot)}(f) = \rhoKR_{p_k(\cdot)}(f) = 2^k \to \infty$. 
  \end{enumerate}
\end{remark}

\section{Muckenhoupt condition for variable exponents}

The Muckenhoupt classes play a fundamental part in the study of weighted Lebesgue spaces. The natural generalization to variable exponent spaces is the following:
\begin{definition}[Muckenhoupt class]
  \label{def:Ap}
  Let $p \in \mathcal{P}(\RRn)$. Then we say that $p \in\mathcal{A}$ if
  \begin{align}
    \label{eq:Aloc}
    [p]_{\mathcal{A}}= 
    \sup_{Q}\frac{\|\indicator_{Q}\|_{p(\cdot)}\|\indicator_{Q}\|_{p'(\cdot)}}{\abs{Q} } < \infty.
  \end{align}
\end{definition}
Note that $[p']_{\mathcal{A}} = [p]_{\mathcal{A}}$. Moreover, $[p]_{\mathcal{A}} \geq \frac 12$ by H\"older's inequality. 
The condition $p\in \mathcal{A}$ is equivalent to the uniform boundedness of the averaging operators $f \mapsto \indicator_Q \dashint f\,dx$, where $Q$ runs over the family of cubes. Indeed, it is shown in \cite[Theorem~4.5.7]{DieHarHasRuz11} that
\begin{align}
  \label{eq:A-vs-TQ}
  \frac 12 [p]_{\mathcal{A}}
  &\leq \sup_{Q} \sup_{f:\norm{f}_{\px}\leq 1} \biggnorm{\indicator_Q \dashint_Q f\,dx}_\px \leq
  2 [p]_{\mathcal{A}}.
\end{align}
The condition~$p \in \mathcal{A}$ appears in \cite{Kop07} as $A_\px$, in \cite[Theorem~4.5.7]{DieHarHasRuz11} as $\mathcal{A}_{\loc}$ and in \cite{Ler10} as $p \in A$.

Additionally, we need the decay condition of Nekvinda~\cite{Nek04}.
\begin{definition}
  We say that $p \in \mathcal{P}(\RRn)$ satisfies the Nekvinda condition $p \in \mathcal{N}$ if there exists $p_\infty \in [1,\infty]$ such that
    $1 \in L^{s(\cdot)}(\mathbb{R}^{n})$ with $
    \tfrac{1}{s(x)} \coloneqq \abs{\tfrac{1}{p(x)} - \tfrac{1}{p_{\infty}}}$.
  We define $[p]_{\mathcal{N} } \coloneqq \norm{1}_{s(\cdot)}$.
\end{definition}
Note that $[p']_{\mathcal{N}} = [p]_{\mathcal{N}}$. Moreover, $[p]_{\mathcal{N}} =\norm{1}_{s(\cdot)} \geq 1$.  This follows from the fact that for all $t>1$, we have $\rho_{s(\cdot)}(t)=\infty$ using $\frac{1}{s(x)}t^{s(x)} \geq e\,\log t$. Also note that $p^- \leq p_\infty \leq p^+$.
Using the above notation we can formulate the embeddings of the sum and the intersection spaces, see \cite[Lemma 3.3.12]{DieHarHasRuz11}:
\begin{lemma}[Lemma 3.3.5 in \cite{DieHarHasRuz11}]
  \label{lem:Nekvinda-minimax}
  Let $p \in \mathcal{N}$ with $p_{\infty}\in [1,\infty]$. Then
  \begin{align*}
    L^{\max\{p(\cdot),\,p_{\infty}\}}(\mathbb{R}^{n}) \hookrightarrow L^{p(\cdot)}(\mathbb{R}^{n}) \hookrightarrow L^{\min\{p(\cdot),\,p_{\infty}\}}(\mathbb{R}^{n})
  \end{align*}
  The embedding constants are at most~$2 [p]_{\mathcal{N}}$. Moreover,
  \begin{align*}
    L^{p(\cdot)}(\mathbb{R}^{n}) \cap L^{p^{+}}(\mathbb{R}^{n}) &\cong L^{p_{\infty}}(\mathbb{R}^{n}) \cap L^{p^{+}}(\mathbb{R}^{n}),
    \\
    L^{p(\cdot)}(\mathbb{R}^{n}) \cap L^\infty(\mathbb{R}^{n}) &\cong L^{p_{\infty}}(\mathbb{R}^{n}) \cap L^\infty(\mathbb{R}^{n}).
  \end{align*}
  The embedding constants are at most~$4[p]_{\mathcal{N}}$.
\end{lemma}
\begin{proof}
  The first part follows from \cite[Lemma~3.3.5]{DieHarHasRuz11} and its proof. The second part follows from \cite[Lemma~3.3.12]{DieHarHasRuz11} and its proof. It is basically a consequence of the first part and Lemma~\ref{lem:embedding}. The third part is a trivial modification.
\end{proof}

In the proof of our main result Theorem~\ref{thm:main}, we want to include unbounded exponents. However, for the proof it is very convenient to work with bounded exponents. This simplifies the notation significantly. Therefore, it is very useful to approximate the exponent $p\in \mathcal{A} \cap \mathcal{N}$ by $p_k \in \mathcal{A} \cap \mathcal{N}$ with $1<p_k^-\leq p_k^+ < \infty$ in such a way that we can later pass back to~$p$. In particular, the approximation of~$p_k$ should respect the $\mathcal{A}$ and $\mathcal{N}$ constants. This is the purpose of the following lemma.
\begin{lemma}[Approximation of variable exponents]
  \label{lem:pkbetter}
  For $p \in \mathcal{P}(\RRn)$ define $p_k \in \mathcal{P}(\RRn)$ with $k \in \setN$ by
  \begin{align}
    \label{eq:def-pk}
    \frac{1}{p_k(x)} &\coloneqq \frac{1}{k+1} + \frac{k-1}{k+1} \frac{1}{p(x)}.
  \end{align}
  Then
  \begin{enumerate}
  \item \label{itm:pkbetter1} $\frac{1}{p_k} \to \frac 1p$ uniformly on~$\RRn$.
  \item \label{itm:pkbetter2} $p_k \to p$ almost everywhere.
  \item \label{itm:pkbetter3} $1+\frac 1k \leq p_k^- \leq p_k^+  \leq k+1$ and $\min \set{p^-,2} \leq p_k^- \leq p_k^+ \leq \max \set{p^+,2}$.
  \item \label{itm:pkbetter4} If $ p \in \mathcal{A}$, then $p_k \in \mathcal{A}$ and $[p_k]_{\mathcal{A}} \leq 8\, [p]_{\mathcal{A}}$.
  \item \label{itm:pkbetter5} If $ p \in \mathcal{N}$, then $p_k \in \mathcal{N}$ and $[p_k]_{\mathcal{N}} \leq 2[p]_{\mathcal{N}}$.
  \item \label{itm:pkbetter6} If $f \in L^2(\RRn)$, then
    \begin{align*}
      \lim_{k \to \infty} \tilde\rho_{p_k(\cdot)}(f) &= \tilde\rho_{p(\cdot)}(f) \qquad \text{and} \qquad 
      \lim_{k \to \infty} \norm{f}_{\tilde{\rho}_{p_k(\cdot)}} =\norm{f}_{\tilde{\rho}_{p(\cdot)}}.
    \end{align*}
  \end{enumerate}
\end{lemma}
\begin{proof}
  The statements~\ref{itm:pkbetter1}--\ref{itm:pkbetter3} follow directly by the definition of the~$p_k$. We continue with~\ref{itm:pkbetter4}. The exponent is defined as a suitable interpolation between~$L^1$, $L^{\px}$ and $L^\infty$, namely, $L^{p_k(\cdot)} = [[L^1,L^\infty]_{\frac 12}, L^\px]_{\frac{k-1}{k+1}}$.  Thus, it follows by the interpolation estimate~\eqref{eq:interpolation}, the classical interpolation estimate for constant exponents, and~\eqref{eq:multiplepower} that
  \begin{align*}
    \norm{\indicator_Q}_{p_k(\cdot)}
    &\leq 2\,\norm{\indicator_Q}_2^{\frac{2}{k+1}}
    \norm{\indicator_Q}_{\px}^{\frac{k-1}{k+1}}
    \\
    &\leq
    2\,
    \bignorm{\indicator_Q}_{1}^{\frac 1{k+1}}
    \bignorm{\indicator_Q}_{\px}^{\frac{k-1}{k+1}}
    \bignorm{\indicator_Q}_{\infty}^{\frac{1}{k+1}}
    \\
    &= 2\,
    \abs{Q}^{\frac 1{k+1}}
    \norm{\indicator_Q}_{\px}^{\frac{k-1}{k+1}}.
  \end{align*}
  Note that $p_k$ is defined such that it commutes with duality, i.e.
  \begin{align}
    \label{eq:pkdual}
    \frac{1}{p_k'(x)} &\coloneqq \frac{1}{k+1} + \frac{k-1}{k+1} \frac{1}{p'(x)}.
  \end{align}
  Thus, by the calculations above we also obtain
  \begin{align*}
    \norm{\indicator_Q}_{p_k'(\cdot)} \leq 2 \abs{Q}^{\frac 1{k+1}} \norm{\indicator_Q}_{\pdx}^{\frac{k-1}{k+1}}.
  \end{align*}
  Combining both estimates we obtain
  \begin{align*}
    \frac{     \norm{\indicator_Q}_{p_k(\cdot)}     \norm{\indicator_Q}_{p_k'(\cdot)}
    }{\abs{Q}}
    &\leq 4 \bigg(\frac{ 
      \norm{\indicator_Q}_{\px}
      \norm{\indicator_Q}_{\pdx}}{\abs{Q}} \bigg)^{\frac{k-1}{k+1}}
    \leq 4\, [p]_{\mathcal{A}}^{\frac{k-1}{k+1}}
    \leq 8\, [p]_{\mathcal{A}}
  \end{align*}
  using in the last step that $[p]_{\mathcal{A}} \geq \frac 12$. Taking the supremum over all cubes~$Q$ proves~\ref{itm:pkbetter4}.

  Suppose now that $p \in \mathcal{N}$. Since $t \mapsto \frac{1}{k+1} + \frac{k-1}{k+1} t$ is a contraction with fixed point~$\frac 12$, we have
  \begin{align*}
    \tfrac 1{s_k(x)} \coloneqq \bigabs{\tfrac 1{p_k(x)} - \tfrac 1{(p_k)_\infty}} \leq 
    \bigabs{\tfrac 1{p(x)} - \tfrac 1{p_\infty}} \eqqcolon \tfrac 1{s(x)},
  \end{align*}
  where $\frac{1}{(p_k)_\infty} = \frac{1}{k+1} + \frac{k-1}{k+1}\frac{1}{p_\infty}$.  
  Thus, $s \leq s_k \leq \infty$ and by Lemma~\ref{lem:embedding} we get
  \begin{align*}
    \norm{1}_{s_k(\cdot)} &\leq 2\, \max \bigset{\norm{1}_{s(\cdot)}, \norm{1}_{\infty}} \leq 2 \norm{1}_{s(\cdot)},
  \end{align*}
  where we used $\norm{1}_{s(\cdot)}= [s]_{\mathcal{N}} \geq 1 = \norm{1}_\infty$.
  This proves~\ref{itm:pkbetter5}.

  It remains to prove~\ref{itm:pkbetter6}. For this let $f \in L^2(\RRn)$. Then $\tilde\rho_2(f) = \frac 12 \norm{f}_2^2<\infty$. It follows directly from \cite[Corollary~3.5.4]{DieHarHasRuz11} that
  \begin{align*}
    \tilde\rho_{p(\cdot)}(f)&\leq \liminf_{k \to \infty} \tilde\rho_{p_k(\cdot)}(f) \qquad \text{and} \qquad 
     \norm{f}_{\tilde{\rho}_{p(\cdot)}} \leq \liminf_{k \to \infty} \norm{f}_{\tilde{\rho}_{p_k(\cdot)}} .
  \end{align*}
  It remains to prove
  \begin{align}
    \label{eq:pxbetter-aux1}
    \limsup_{k \to \infty} \tilde\rho_{p_k(\cdot)}(f) &\leq     \tilde\rho_{p(\cdot)}(f),
    \\
    \label{eq:pxbetter-aux2}
    \limsup_{k \to \infty} \norm{f}_{\tilde{\rho}_{p_k(\cdot)}}&\leq \norm{f}_{\tilde{\rho}_{p(\cdot)}}.
  \end{align}
  We start with the proof of~\eqref{eq:pxbetter-aux1}. For this we can assume that~$\tilde\rho_\px(f)< \infty$.  Let us define $q,r \in \mathcal{P}(\RRn)$ by $q \coloneqq \min \set{p,2}$ and $r \coloneqq \max \set{p,2}$. It follows from $\abs{\frac{1}{p_k(x)} - \frac 12}= \frac{k-1}{k+1} \abs{\frac{1}{p(x)}-\frac 12}$ that $q \leq p_k \leq r$. Moreover,
  \begin{align*}
    \tilde\rho_{q(\cdot)}(f)
    &=
    \tilde\rho_{2}(\indicatorset{p\geq 2}f) +
    \tilde\rho_{p(\cdot)}(\indicatorset{p< 2}f)
    \leq \tilde\rho_2(f) + \tilde\rho_{\px}(f) < \infty,
    \\
    \tilde\rho_{r(\cdot)}(f)
    &=
    \tilde\rho_2(\indicatorset{p\leq 2}f) +
    \tilde\rho_{p(\cdot)}(\indicatorset{p> 2}f)
    \leq \tilde\rho_2(f) + \tilde\rho_{\px}(f) < \infty.
  \end{align*}
  Thus,~\eqref{eq:pxbetter-aux1} follows by use of Lemma~\ref{lem:approxpxmodular}. This proves~\eqref{eq:pxbetter-aux1}.

  It remains to prove~\eqref{eq:pxbetter-aux2}. Let $0<\lambda<1$, then
  \begin{align*}
    \tilde\rho_{q(\cdot)}\bigg(\frac{\lambda f}{\norm{f}_{\tilde{\rho}_{p(\cdot)}}}\bigg)
    &=
    \tilde\rho_2\bigg(\frac{\indicatorset{p\geq 2} \lambda f}{\norm{f}_{\tilde{\rho}_{p(\cdot)}}}\bigg) +
    \tilde\rho_\px\bigg(\frac{\indicatorset{p> 2}\lambda f}{\norm{f}_{\tilde{\rho}_{p(\cdot)}}}\bigg)
    \leq\frac{\lambda^2\tilde\rho_2(f)}{\norm{f}_{\tilde{\rho}_{p(\cdot)}}^2} + 1 < \infty,
    \\
    \tilde\rho_{r(\cdot)}\bigg(\frac{\lambda f}{\norm{f}_{\tilde{\rho}_{p(\cdot)}}}\bigg)
    &=
    \tilde\rho_2\bigg(\frac{\indicatorset{p\leq 2} \lambda f}{\norm{f}_{\tilde{\rho}_{p(\cdot)}}}\bigg) +
    \tilde\rho_\px\bigg(\frac{\indicatorset{p< 2}\lambda f}{\norm{f}_{\tilde{\rho}_{p(\cdot)}}}\bigg)
    \leq\frac{\lambda^2\tilde\rho_2(f)}{\norm{f}_{\tilde{\rho}_{p(\cdot)}}^2} + 1 < \infty.
  \end{align*}
  Thus, it follows by Lemma~\ref{lem:approxpxmodular} that
  \begin{align*}
    \limsup_{k \to \infty} \tilde\rho_{p_k(\cdot)} \bigg( \frac{\lambda f}{\norm{f}_{\tilde{\rho}_{p_k(\cdot)}}}\bigg) = \tilde\rho_{p(\cdot)} \bigg( \frac{\lambda f}{\norm{f}_{\tilde{\rho}_{p(\cdot)}}}\bigg) \leq \lambda \tilde\rho_{p(\cdot)} \bigg( \frac{f}{\norm{f}_{\tilde{\rho}_{p(\cdot)}}}\bigg) \leq \lambda < 1.
  \end{align*}
  Thus $\tilde\rho_{p_k(\cdot)}( \frac{\lambda f}{\norm{f}_{\tilde{\rho}_{\px}}}) \leq 1$ for large~$k$, which implies $\lambda \norm{f}_{\tilde{\rho}{p_k(\cdot)}} \leq \norm{f}_{\tilde{\rho}_{\px}}$ for large~$k$. Since $\lambda \in (0,1)$ was arbitrary, this proves~\eqref{eq:pxbetter-aux2}. This concludes the proof of the lemma.
\end{proof}

\section{Proof of the main result}
\label{sec:proof-main-result}

In this section we prove our main result Theorem~\ref{thm:main}. We begin in Section~\ref{sec:using-local-a_infty} first with an auxiliary result that shows that $t^{p'(\cdot)}$ is locally an~$\mathcal{A}_\infty$-weight. This is then used in Section~\ref{sec:proof-main-theorem} to prove our main result by means of \Calderon{}-Zygmund argument. The proof is based on the ideas of~\cite{Ler10} and~\cite[Theorem~4.52]{CruFio13}. However, we adapted the proof so that the constants no longer depend on~$p^+$, which allows us to treat the general case.

\subsection{\texorpdfstring{Using local $\mathcal{A}_\infty$-weights}{Using local A-infinity-weights}}
\label{sec:using-local-a_infty}

Let $Q\subset \RRn$ be a cube. Then by $\abs{Q}$ we denote its Lebesgue measure.  For $f \in L^1(Q)$ we define $|f|_Q \coloneqq \dashint_Q \abs{f}\,dx = \frac{1}{\abs{Q}} \int_Q \abs{f}\,dx$. A weight~$w$ is a non-negative function on~$\RRn$ or~$Q$, which is locally integrable. We define $w(E) \coloneqq \int_E w(x)\,dx$ for any measurable set~$E$. By $B_r(x)$ we denote the open ball in~$\RRn$ with radius~$r$ and center~$x$. For a measurable set~$E$ we denote by~$\indicator_E$ its indicator function.

In the following we need the notion of (local) $\mathcal{A}_\infty$-weights.

\begin{definition}
  \label{def:Ainftyloc}
  Let~$w$ be a weight on a cube~$Q \subset \RRn$. We say that~$w \in \mathcal{A}_\infty(Q)$ if there exists $\alpha,\beta\in(0,1)$ such that for any cube $Q'\subset Q$ and for any measurable subset $E\subset Q'$ the following holds
  \begin{align}
    \label{eq:weight-alpha-beta}
    |E|\geq\alpha|Q'| \quad\text{implies}\quad w(E)\geq\beta w(Q').
  \end{align}
\end{definition}
It is well known fact  that the class $A_{\infty}$ can be defined in many equivalent ways. 
In particular, $w\in A_{\infty}(Q)$ if and only if  there
exist constants $c_0,\varepsilon>0$ such that for any cube $Q'\subset Q$ and for any measurable subset $E\subset Q'$,
\begin{equation}\label{eq:weight-c-epsilon}
  \frac{w(E)}{w(Q')}\leq c_0\left(\frac{|E|}{|Q'|}\right)^{\varepsilon}.
\end{equation}
The constants in \eqref{eq:weight-c-epsilon} depend only on the dimension $n$ and the constants in \eqref{eq:weight-alpha-beta}, see \cite[Theorem 7.3.3]{Grafakos}. These results are stated for ~$\RRn$ but transfer immediately to those on~$Q$.

The following result is a refinement of~\cite[Lemma~3.1]{Ler10}, where we carefully keep track of the dependence of the constants. In order to pass later to the general case with $p^-=1$ or $p^+=\infty$, it is essential that the constants in Lemma~\ref{lem:avgmodular} are independent of~$p^-$ and $p^+$.
\begin{lemma}
  \label{lem:avgmodular}
  Let $p\in \mathcal{A}$ with $1<p^-\leq p^+<\infty$ and
    $\lambda \coloneqq \min \set{\tfrac 12, \tfrac{1}{2[p]_{\mathcal{A}}}}$.
  Then for every cube $Q$, every $f \in L^{\px}(Q)$ with $|f|_{Q}\geq 1$ and $\|f\|_{p(\cdot)}\leq \frac{1}{2}$ we have
  \begin{align*}
    \int_{Q}  (\lambda |f|_{Q})^{p(x)}dx &\leq \int_{Q}|2f(x)|^{p(x)}\,dx
  \end{align*}
\end{lemma}
\begin{proof}
  By H\"{o}lder's inequality we have
  \begin{align}\label{eq:3.1}
    \dashint_{Q}|f(x)|dx \leq \frac{2\|f\|_{p(\cdot)}\|\indicator_{Q}\|_{p'(\cdot)}}{\abs{Q}} \leq \frac{\|\indicator_Q\|_{p'(\cdot)}}{\abs{Q}}.
  \end{align}
  By continuity we find $\alpha>0$ such that
  \begin{align}\label{eq:3.2}
    \dashint_{Q}\alpha^{p'(y)-1}dy=\dashint_{Q}|f(x)|dx \leq  \frac{\|\indicator_Q\|_{p'(\cdot)}}{\abs{Q}}.
  \end{align}
  Using $(p')^+<\infty$ and the unit ball property we obtain
  \begin{align}\label{3.3}
    \dashint_{Q}\bigg(\frac{1}{\|\indicator_{Q}\|_{p'(\cdot)}} \bigg)^{p'(y)-1} \!\!\!dy
    =   \norm{\indicator_Q}_\px  \dashint_{Q}\left(\frac{1}{\|\indicator_Q\|_{p'(\cdot)}} \right)^{p'(y)}dy = \frac{\|\indicator_Q\|_{p'(\cdot)}}{\abs{Q}}.
  \end{align}
  Thus, by strict monotonicity of~\eqref{eq:3.2} in~$\alpha$ we obtain $\alpha\leq \frac{1}{\|\indicator_Q\|_{p'(\cdot)}}$. Since $|f|_Q \geq 1$, we obtain also by the  equality in~\eqref{eq:3.2} that $\alpha\geq 1$. Overall, we have
  \begin{align}\label{eq:3.4}
    1 \leq \alpha\leq \frac{1}{\|\indicator_Q\|_{p'(\cdot)}}.
  \end{align}
  We calculate
  \begin{align}\label{eq:3.5}
    \begin{aligned}
      \int_{Q}(\lambda |f|_{Q})^{p(x)}dx &=\int_{Q}\left( \lambda \dashint_{Q}\alpha^{p'(y)-1}dy\right)^{p(x)}dx
      \\
      &=\int_{Q}\left( \lambda \dashint_{Q}\alpha^{p'(y)-p'(x)}dy\right)^{p(x)} \alpha^{p'(x)}\,dx
      \\
      &= \lambda\,\left(\frac{1}{|Q|}\int_{Q}\left( \lambda \dashint_{Q}\alpha^{p'(y)-p'(x)}dy\right)^{p(x)-1}dx\right)\int_{Q}\alpha^{p'(y)}dy.
    \end{aligned}
  \end{align}
  For each $x\in Q$ we consider the partition of  $Q$ into
  $E_{1}(x)=\{y\in Q:\, p'(y)>p'(x)\}$ and $E_{2}(x)=Q\backslash E_{1}(x)$. Using \eqref{eq:3.4} and the fact that  $\alpha\geq1$, we obtain
  \begin{align*}
    \dashint_{Q}\alpha^{p'(y)-p'(x)}dy
    &=\frac{1}{\abs{Q}}\left(\int_{E_{1}(x)} \alpha^{p'(y)-p'(x)}dy+\int_{E_{2}(x)}\alpha^{p'(y)-p'(x)}dy\right)
    \\
    &\leq \frac{\| \indicator_Q\|^{p'(x)}_{p'(\cdot)}}{\abs{Q}} \int_{E_{1}(x)}\left[ \frac{1}{\| \indicator_Q\|_{p'(\cdot)}} \right] ^{p'(y)}dy+ 1
    \\
    &\leq \frac{\|\indicator_Q\|_{p'(\cdot)}^{p'(x)}}{\abs{Q}}+1.
  \end{align*}
  Hence, we obtain 
  \begin{align*}
    \lefteqn{\textrm{I} \coloneqq \frac{1}{|Q|}\int_{Q}\left(\lambda\dashint_{Q}\alpha^{p'(y)-p'(x)}dy\right)^{p(x)-1}dx}
    \qquad &
    \\
    &\leq\frac{1}{|Q|}\int_{Q}\lambda^{p(x)-1} \left(\frac{\|\indicator_Q\|^{p'(x)}_{p'(\cdot)}}{\abs{Q}}+1 \right)^{p(x)-1}dx
    \\
    &\leq \frac{1}{|Q|} \int_{Q}\lambda^{p(x)-1} 2^{p(x)-1} \left(\left[\frac{\|\indicator_Q\|^{p'(x)}_{p'(\cdot)}}{\abs{Q}}\right]^{p(x)-1}+1\right)dx
    \\
    &\leq  \frac{1}{2\lambda} \int_{Q}(2\lambda)^{p(x)}\left[\frac{\|\indicator_Q\|_{p'(\cdot)}}{\abs{Q}}\right]^{p(x)}dx+1.
  \end{align*}
  Due to the fact that $p\in \mathcal{A} $ and $2 \lambda [p]_{\mathcal{A}} \leq 1$, we have
  \begin{align}
    \label{eq:aux-I}
    \begin{aligned}
      \textrm{I} &\leq \frac{1}{2\lambda} \int_{Q}(2\lambda [p ]_{\mathcal{A}})^{p(x)} \frac{1}{\norm{\indicator_Q}^{p(x)}_{\px}} dx+1
      \leq \frac{1}{2\lambda} \int_Q \frac{1}{\norm{\indicator_Q}^{p(x)}_{\px}} dx+1
      \leq \frac{1}{2\lambda}+1
    \end{aligned}
  \end{align}
  Also we have the following estimate.
  \begin{align}\label{eq:3.7}
    \begin{aligned}
      \int_{Q}\alpha^{p'(y)}\,dy &=2\alpha\int_{Q}|f(y)|\,dy- \int_{Q}\alpha^{p'(y)}\,dy
      \\
      &\leq \int_{Q}\alpha^{p'(y)}\,dy+\int_{Q}|2f(y)|^{p(y)}\,dy-\int_{Q}\alpha^{p'(y)}\,dy
      \\
      & = \int_{Q}|2f(y)|^{p(y)}\,dy
    \end{aligned}
  \end{align}
  Now, \eqref{eq:3.5}, \eqref{eq:aux-I} and~\eqref{eq:3.7} and $\lambda \leq \frac 12$ imply
  \begin{align*}
    \int_{Q}(\lambda |f|_{Q})^{p(x)}dx &\leq \lambda\, \bigg( \frac{1}{2\lambda} + 1\bigg) \int_{Q}|2f(y)|^{p(y)}\,dy \leq \int_{Q}|2f(y)|^{p(y)}\,dy.
  \end{align*}
  This proves the claim.
\end{proof}

We are now ready to prove that the function $t^{p'(\cdot)}$ is under certain assumptions a local~$\mathcal{A}_\infty$ weight. Note that the following lemma will later be applied to the dual exponent~$p'$. To avoid confusion, we formulate it therefore for a variable exponent~$q$.
\begin{lemma}\label{lem:Ainfinity}
  Let $q\in \mathcal{A}$ with $1<q^-\leq q^+<\infty$.  Suppose that $1\leq t\leq \frac{2[q]_{\mathcal{A}}}{\|\indicator_Q\|_{q(\cdot)}}$. Then $t^{q(\cdot)}\in A_{\infty}(Q)$, with the $A_{\infty}$ constant depending only on $q^+$ and ~$[q]_{\mathcal{A}}$.
\end{lemma}	  
\begin{proof}
  Let us take $\lambda \coloneqq\min\set{ \tfrac{1}{2},\tfrac{1}{2[q]_{\mathcal{A}}}}$.  Let $Q' \subset Q$ be a cube and $E\subset Q'$ with $|E|\geq|Q'|/2$. We define $f \coloneqq \frac{t}{4[q]_{\mathcal{A}}} \indicator_E$.  Then $\norm{f}_{q(\cdot)} \leq \frac{t}{4[q]_{\mathcal{A}}} \norm{\indicator_Q}_{q(\cdot)} \leq\frac 12$.
  Moreover,
  \begin{align*}
    |f|_{Q'} \geq \frac{\abs{E}}{\abs{Q'}}\frac{t}{4[q]_{\mathcal{A}}} \geq \frac{t}{8[q]_{\mathcal{A}}}.
  \end{align*}
  Suppose that $t \geq 8[q]_{\mathcal{A}}$, then $\abs{f}_{Q'}\geq 1$ and by Lemma~\ref{lem:avgmodular} we have
  \begin{align*}
    \int\limits_{Q'}\! \bigg( \frac{\lambda\, t}{8[q]_{\mathcal{A}}}\bigg)^{q(x)} \!\!dx \leq  \int\limits_{Q'}\! (\lambda\abs{f}_{Q'})^{q(x)}\!dx\leq\! \int\limits_{Q'} \abs{2f}^{q(x)}\!dx = \int\limits_{E} \!\bigg( \frac{t}{2[q]_{\mathcal{A}}}\bigg)^{q(x)}\!dx \leq \int\limits_{E} t^{q(x)}dx,
  \end{align*}
  which implies
  \begin{align}
    \label{eq:Ainfinity-aux1}
    \int_{Q'}t^{q(x)}dx\leq \left(\frac{8[q]_{\mathcal{A}}}{\lambda}\right)^{q^{+}}\int_{E}t^{q(x)}dx \leq \big(32[q]^2_{\mathcal{A}}\big)^{q^+}\int_{E}t^{q(x)}dx.
  \end{align}
  On the other hand it $1 \leq t \leq 8[q]_{\mathcal{A}}$, then
  \begin{align}
    \label{eq:Ainfinity-aux2}
    \int_{Q'}t^{q(x)}dx\leq \big(8[q]_{\mathcal{A}}\big)^{q^+} \abs{Q'} \leq
    2\big(8[q]_{\mathcal{A}}\big)^{q^+} \int_E t^{q(x)}\,dx.
  \end{align}
  Thus, \eqref{eq:weight-alpha-beta} holds  with $\alpha=\frac 12$ and $\beta = \max \set{(32[q]^2_{\mathcal{A}})^{q^+}, 2(8[q]_{\mathcal{A}})^{q^+}}$. Hence, $t^{q(\cdot)}\in A_{\infty}(Q)$ with $\mathcal{A}_\infty$-constant only depending on~$q^+$ and $[q]_{\mathcal{A}}$.
\end{proof}

\subsection{Proof of the main theorem}
\label{sec:proof-main-theorem}

In this section we will prove our main result Theorem~\ref{thm:main}. The proof is split into two parts. First, we prove the following proposition, which is exactly Theorem~\ref{thm:main} under the additional assumption that~$p^+< \infty$. Then we will deduce Theorem~\ref{thm:main} by a suitable approximation process using the results of Section~\ref{sec:limits-funct-expon}.
\begin{proposition}
  \label{pro:main-pre}
  Let $p\in \mathcal{A}\cap\mathcal{N}$ with $1<p^-\leq p^+ < \infty$. Then there exists a constant $c>0$ depending on $(p')^+,\,[p]_{\mathcal{A}},\,[p]_{\mathcal{N}}$ and $n$ such that for any $f\in L^{p(\cdot)}(\mathbb{R}^{n})$
  \begin{align}
    \label{eq:maininequality-pre}
    \norm{Mf}_{p(\cdot)}\leq c\,\norm{f}_{p(\cdot)}.
  \end{align}
\end{proposition}
This proposition is a refinement of \cite[Theorem~4.52]{CruFio13} which allows for a better tracking of the constants. Different from \cite[Theorem~4.52]{CruFio13} our constants do not depend on~$p^+$. This allows us to include the case~$p^+=\infty$ using the approximation results of Section~\ref{sec:limits-funct-expon}. However, the assumption~$p^+<\infty$ helps in the proof of Proposition~\ref{pro:main-pre} to avoid technical problems.
\begin{proof}[Proof of Theorem~\ref{thm:main}]
  Suppose that $p \in \mathcal{A} \cap \mathcal{N}$ with $1 < p^- \leq p^+ \leq \infty$ and $f \in L^\px(\RRn)$ as in Theorem~\ref{thm:main}. Without loss of generality we can assume that~$f \geq 0$. For $m \in \mathbb{N}$, we define $f_m \coloneqq \indicator_{B_m(0)} \min \{ f, m \}$.
   Moreover, for $k \in \setN$, let $p_k \in \mathcal{A} \cap \mathcal{N}$  be the approximation of~$p$ as in Lemma~\ref{lem:pkbetter}. Then with $1+\frac 1k <p_k^- \leq p_k^+ < k+1$, $[p_k]_{\mathcal{A}} \leq 8[p]_{\mathcal{A}}$, $[p_k]_{\mathcal{N}} \leq 2 [p]_{\mathcal{N}}$ and $(p_k')^+ \leq \max \set{(p')^+,2}$. Thus, by Proposition~\ref{pro:main-pre} and the norm equivalence in~\eqref{eq:norm-equivalence}
  we obtain
  \begin{align*}
    \norm{Mf_m}_{\tilde{\rho}_{p_k(\cdot)}}\leq c\,\norm{f_m}_{\tilde{\rho}_{p_k(\cdot)}},
  \end{align*}
  where the constant only depends on $(p')^+,\,[p]_{\mathcal{A}},\,[p]_{\mathcal{N}}$. 

  By construction of $f_m$ we have  $f_m \in L^2(\RRn)$ and therefore $Mf_m \in L^2(\RRn)$. So by Lemma~\ref{lem:pkbetter}~\ref{itm:pkbetter6} we obtain with $k \to \infty$
  \begin{align*}
    \norm{Mf_m}_{\tilde{\rho}_{p(\cdot)}}\leq c\,\norm{f_m}_{\tilde{\rho}_{p(\cdot)}},
  \end{align*}
  Also by choice of~$f_m$ we have $f_m \nearrow f$ and therefore $M f_m \nearrow M f$. Thus, by the Fatou property in Lemma~\ref{lem:fatou} we can pass to the limit~$m \to \infty$ and obtain
  \begin{align*}
    \norm{Mf}_{\tilde{\rho}_{p(\cdot)}}\leq 2c\,\norm{f}_{\tilde{\rho}_{p(\cdot)}},
  \end{align*}
  This and the norm equivalence in~\eqref{eq:norm-equivalence} prove Theorem~\ref{thm:main}.
\end{proof}
It remains to prove Proposition~\ref{pro:main-pre}. Let us recall a few basic facts that we will use. As in~\cite{HytLacPer13,Ler17} a \emph{dyadic grid}~$\mathcal{D}$ is a collection of cubes from~$\RRn$ with the following properties
\begin{enumerate}
\item for any $Q\in \mathcal{D}$ its side-length $l_{Q}$ is of the form $2^{m},\,m\in \mathbb{Z};$
\item  $Q_1\cap Q_2\in \{Q_1,Q_2,\emptyset\}$ for any $Q_1,Q_2\in \mathcal{D};$
\item for every $m\in \mathbb{Z}$ the $Q \in \mathcal{D}$ with sidelength $2^{m}$ form a partition  of $\mathbb{R}^{n}$.
\end{enumerate}
For a dyadic grid~$\mathcal{D}$ we denote the associated dyadic maximal operator $M^{\mathcal{D}_\alpha}$ by
\begin{align*}
  M^{\mathcal{D}}f(x)=\sup_{x\in Q,Q\in\mathcal{D}}\frac{1}{|Q|}\int_{Q}|f(y)|dy.
\end{align*}
It has been shown in~\cite{HytLacPer13} that for every $\alpha \in \set{0,\frac 13, \frac 23}^n$ the family
\begin{align*}
  \mathcal{D}_\alpha\coloneqq \bigset{2^{-m}([0,1)^{n}+j+(-1)^m\alpha),\,\, m\in \mathbb{Z},\,\, j\in\mathbb{Z}^{n}}
\end{align*}
is a dyadic grid. Moreover, it immediately follows from \cite[Lemma 2.5]{HytLacPer13} that
\begin{align}
  \label{eq:2.2}
  Mf(x)\leq 6^{n}\sum_{\alpha\in \set{0,\frac 13,\frac 23}^n}M^{\mathcal{D}_{\alpha}}f(x).
\end{align}

The following lemma is a standard variation of the Calder\'on-Zygmund decomposition, see \cite[Lemma 2.4]{Ler17}.
\begin{lemma}\label{lem:dyadic}
  Suppose $\mathcal{D}$ is a dyadic grid. Let $f\in L^{p}(\mathbb{R}^{n}),\,1\leq p<\infty,$ and $\gamma>1$. For $k \in \setZ$ define
  \begin{align*}
    \Omega_{k}=\{x\in \mathbb{R}^{n}\,:\,M^{\mathcal{D}}f(x)>\gamma^{k}\}.
  \end{align*}
  Then each $\Omega_k$ can be written as a union of pairwise disjoint cubes $Q_j^{k}\in \mathcal{D}$ satisfying
  \begin{align}\label{eq:dyadic1}
    |Q_{j}^{k}\cap \Omega_{k+l}|\leq 2^{n} \gamma^{-l} |Q_{j}^{k}| \qquad \text{for all $l\in \setN$.}
  \end{align}
  Moreover, for each of those~$Q^k_j$ we have
  \begin{align}
    \label{eq:dyadic2}
    \gamma^k < \abs{f}_{Q_j^k} \leq 2^n \gamma^k.
  \end{align}
\end{lemma}

Let us now turn to the proof of Proposition~\ref{pro:main-pre}.

\begin{proof}[Proof of Proposition~\ref{pro:main-pre}]
  Let $p\in \mathcal{A}\cap\mathcal{N}$ with $1<p^-\leq p^+ < \infty $. Without loss of generality we can assume that~$f \geq 0$. By a simple approximation argument it suffices to prove~\eqref{eq:maininequality-pre} for bounded functions with compact support: Indeed, if $f \in L^\px(\RRn)$, then $f_m \coloneqq \indicator_{B_m(0)} \min \set{f,m}$ approximates~$f$ by bounded functions with compact support. Now, $f_m \nearrow f$, $M f_m \nearrow Mf$ and the Fatou property Lemma~\ref{lem:fatou} transfer the estimates for $f_m$ to $f$. Hence, we can assume in the following that $f$ is bounded, has compact support and $f \geq 0$.

  Due to~\eqref{eq:2.2} it suffices to prove that for each dyadic grid~$\mathcal{D}$ we have
  \begin{align}
    \label{eq:main-dyadic}
    \norm{M^{\mathcal{D}} f}_{\px} &\leq c \norm{f}_{\px}.
  \end{align}
  
  Due to the norm conjugate formula of Lemma~\ref{lem:norm-formula} it suffices to prove the existence of~$c>0$, which only depends on $(p')^+$, $[p]_{\mathcal{A}}$ and $[p]_{\mathcal{N}}$, such that
  \begin{align}
    \label{eq:3.8}
    \int_\RRn (M^{\mathcal{D}}f)(x)g(x)dx\leq c\,\norm{f}_{p(\cdot)}
  \end{align}
  for all $g \in L^\pdx(\RRn)$ with $\norm{g}_\pdx \leq \frac 12$ and $g\geq 0$, which are bounded and have compact support.

  For each $k\in \setZ$ set
  \begin{align*}
    \Omega_{k}=\{x\in\mathbb{R}^{n}\,\,:\,\,M^{\mathcal{D} }f(x)>3^{nk}\}
  \end{align*}
  and $D_{k}\coloneqq\Omega_{k}\backslash\Omega_{k+1}$. Then, $\bigcup_k D_k = \{ x : M^{\mathcal{D}} f(x) > 0 \}$ up to a set of Lebesgue measure zero. By Lemma \ref{lem:dyadic} each $\Omega_{k}$ can be written as a union of pairwise disjoint cubes $Q_{j}^{k}\in \mathcal{D}$ and $3^{nk} < f_{Q^k_j} \leq 2^n 3^{nk}$. Given $k \in \setZ$ the index~$j$ will run over the indices of those cubes.  Let us define the sets $E_{j}^{k}=Q_{j}^{k}\cap D_{k}$. Then the $E_{j}^{k}$ are pairwise disjoint and $\bigcup_{j}E_{j}^{k}=D_{k}.$

  Define
  \begin{align*}
    Tg(x)\coloneqq \sum_{k \in \setZ}\sum_{j}\left(\frac{1}{|Q_{j}^{k}|}\int_{E_{j}^{k}}gdx\right)\indicator_{Q_j^{k}}(x).
  \end{align*}
  Using the above definitions we get
  \begin{align*}
    \int_\RRn(M^{\mathcal{D}}f)(x)g(x)\,dx &= \sum_{k \in \setZ}
    \int_{D_k}(M^{\mathcal{D}}f)(x)g(x)\,dx
    \\
    &\leq 3^n \sum_{k \in \setZ}
    \int_{D_k} 3^{nk}g(x)\,dx
    \\
    &= 3^n \sum_{k \in \setZ} \sum_j
    \int_{E_j^k} 3^{nk}g(x)\,dx
    \\
    &\leq 3^{n}\sum_{ k\in \setZ}\sum_{j}f_{Q_{j}^{k}}\int_{E_{j}^{k}}g\,dx
    \\
    &=3^{n}\int_{\mathbb{R}^{n}}f\,Tg \,dx
    \\
    &\leq
    2\cdot 3^{n}\|f\|_{p(\cdot)}\|Tg\|_{p'(\cdot)}.
  \end{align*}
  Consequently, to prove \eqref{eq:3.8}, it is sufficient to show that
  \begin{align}\label{eq:avergeneral}
    \|Tg\|_{p'(\cdot)}\leq c,
  \end{align}
  where the constant in \eqref{eq:avergeneral} depends only on $(p')^+,\,[p]_{\mathcal{A}}$ and $[p]_\mathcal{N}$.
  
  Since $Q_{j}^{k} \subset \Omega_{k} = \bigcup_{l=0}^{\infty} D_{k+l}$ up to a set of Lebesgue measure zero, we have that  
  $Tg = \sum_{l=0}^{\infty} T_{l}g$,  
  where
  \begin{align}\label{eq:avgtl}
    T_{l}g(x) = \sum_{k \in \mathbb{Z}} \sum_{j} \alpha_{j,k}(g) \, \indicator_{Q_{j}^{k} \cap D_{k+l}}(x) \qquad \text{for all $l \geq 0$}.
  \end{align}
  here, $\alpha_{j,k}(g)$ is defined as $\alpha_{j,k}(g) \coloneqq \frac{1}{|Q_{j}^{k}|} \int_{E_{j}^{k}} g \, dx.$
  We split the operator \eqref{eq:avgtl} into two parts and prove the boundedness for each part separately.
  
  Let $\mathcal{I}_{1} = \{(j, k) : \alpha_{j,k}(g) > 1\}$ and $\mathcal{I}_{2} = \{(j, k) : \alpha_{j,k}(g) \leq 1\}$.  For $l \geq 0$ and $m = 1, 2$ define
  \begin{align*}
    T_{l}^{(m)}g(x) \coloneqq \sum_{(j, k) \in \mathcal{I}_{m}} \alpha_{j,k}(g) \, \indicator_{Q_{j}^{k} \cap D_{k+l}}(x).
  \end{align*}
  By H\"older's inequality, $\norm{g}_\pdx \leq \frac 12$ and the condition $p \in \mathcal{A}$, we obtain
  \begin{align}
    \label{eq:aux1}
    \alpha_{j,k}(g) = \frac{1}{\abs{Q^k_j}} \int_{E^k_j} g\,dx \leq\frac{1}{|Q_{j}^{k}|}\|\indicator_{E_{j}^{k}}\|_{p(\cdot)}
    \leq\frac{1}{|Q_{j}^{k}|}\|\indicator_{Q_{j}^{k}}\|_{p(\cdot)}\leq
    \frac{[p']_{\mathcal{A}}}{\|\indicator_{Q_{j}^{k}}\|_{p'(\cdot)}}.
  \end{align}
  We begin with the estimate of $T_l^{(1)}g$. In this case we need the indices $(j,k) \in \mathcal{I}_1$ with~$\alpha_{j,k}(g)>1$. Due to~\eqref{eq:aux1} we can apply Lemma~\ref{lem:Ainfinity} to $q(\cdot) = p'(\cdot)$ to obtain that $\alpha_{j,k}(g)^{p'(x)}\in \mathcal{A}_{\infty}(Q_{j}^{k})$, where the $\mathcal{A}_\infty$-constant only depends on $[p]_{\mathcal{A}}$ and $(p')^+$. Thus, there exists~$\epsilon>0$ and $c_0 \geq 1$ such that~\eqref{eq:weight-c-epsilon} holds for $w(x)=\alpha_{j,k}(g)^{p'(x)}$ with $\epsilon$ and $c_0$ only depending on $[p]_{\mathcal{A}}$ and $(p')^+$ (in particular independent of~$j$ and $k$).  Hence, we have
  \begin{align}
    \int_{Q_{j}^{k}\cap D_{k+l}}\alpha_{j,k}(g)^{p'(x)}dx
    &\leq c_0\left(\frac{|Q_{j}^{k}\cap D_{k+l}|}{|Q_{j}^{k}|}\right)^{\varepsilon}\int_{Q_{j}^{k}}\alpha_{j,k}(g)^{p'(x)}dx
  \end{align}
  Let $\lambda = \min \set{\frac12, \frac{1}{2[p]_{\mathcal{A}}}}$. Then with Lemma~\ref{lem:avgmodular} (applied to~$p'$) and~\eqref{eq:dyadic1} we estimate
  \begin{align}
    \label{eq:3.9}
    \int_{Q_{j}^{k}\cap D_{k+l}}\big(\lambda\alpha_{j,k}(g)\big)^{p'(x)}dx
    &\leq c_0\left(\frac{|Q_{j}^{k}\cap D_{k+l}|}{|Q_{j}^{k}|}\right)^{\varepsilon} \int_{E_{j}^{k}}(2g(x))^{p'(x)}dx
    \\
    &\leq c_0 2^{\epsilon n} 3^{-l\epsilon n}  \int_{E_{j}^{k}}(2g(x))^{p'(x)}dx.
  \end{align}
  Using this and $\norm{g}_{\pdx} \leq \frac 12$ we obtain
  \begin{align}\label{eq:biggerone}
    \begin{aligned}
      \int_{\mathbb{R}^{n}}(\lambda T^{(1)}_{l}g(x))^{p'(x)}dx
      &=\sum_{(j,k)\in \mathcal{I}_{1}}\int_{Q_{j}^{k}\cap D_{k+l}}\big(\lambda \alpha_{j,k}(g)\big)^{p'(x)}dx
      \\
      &\leq c_02^{\epsilon n} 3^{-l\epsilon n}\sum_{(j,k)\in\mathcal{I}_{1}}\int_{E_{j}^{k}}(2g(x))^{p'(x)}dx
      \\
      &\leq c_02^{\epsilon n} 3^{-l\epsilon n}\int_{\mathbb{R}^{n}}(2g(x))^{p'(x)}dx
      \\
      &\leq  c_02^{\epsilon n} 3^{-l\epsilon n}.
    \end{aligned}
  \end{align}
  From this, we obtain with~\eqref{eq:trick} 
  \begin{align}\label{eq:Aver1}
  	\|T^{(1)}_{l} g\|_{p'(\cdot)} \leq \lambda^{-1} c_0 2^{\epsilon n} \big(3^{-l\epsilon n}\big)^{\frac{1}{(p')^+}} \leq 4[p]_{\mathcal{N}} c_0 2^{\epsilon n} \big(3^{\frac{-\epsilon n}{(p')^+}}\big)^l.
  \end{align}
  
  Let us turn to the estimate of $T_l^{(2)}(g)$. We decompose~$g$ into $g=g_1+g_2$ with
  $g_{1} = g \indicator_{\{x : g(x) \leq 1\}}$ and $g_{2} = g \indicator_{\{x : g(x) > 1\}}$. Then
  \begin{align*}
    T_{l}^{(2)} g(x) = T_{l}^{(2)} g_{1}(x) + T_{l}^{(2)} g_{2}(x).
  \end{align*}
  We need to estimate the integral $\int_{\mathbb{R}^{n}}|T^{(2)}_{l}g_{1}|^{P}dx$, for  $P=p'_{\infty}$  and $P=(p')^{+}.$
  Then by Jensen's inequality and~\eqref{eq:dyadic1} we have
  \begin{align}
    \label{eq:aux3}
    \begin{aligned}
      \int_{\mathbb{R}^{n}}|T^{(2)}_{l}g_{1}(x)|^{P}dx&=\sum_{(j,k)\in \mathcal{I}_{2}}\int_{Q_{j}^{k}\cap D_{k+l}}(\alpha_{j,k}(g_{1}))^{P}dx
      \\
      &=\sum_{(j,k)\in \mathcal{I}_{2}} \abs{Q_j^k \cap D_{k+l}} \bigg( \frac{1}{\abs{Q^k_j}} \int_{E_{j,k}} g_{1}\,dx \bigg)^P
      \\
      &\leq\sum_{(j,k)\in \mathcal{I}_{2}}\frac{|Q_{j}^{k}\cap D_{k+l}|}{|Q_{j}^{k}|}\int_{E_{j,k}}|g_{1}(x)|^{P}dx
      \\
      &\leq 2^n 3^{-nl}\int_{\mathbb{R}^{n}}|g_{1}|^{P}dx.
    \end{aligned}
  \end{align}
  
  Let us first consider the case when $P=p'_\infty$. Define $s \in \mathcal{P}(\RRn)$ by $\frac{1}{s} \coloneqq \abs{\frac 1{p'} - \frac 1{p'_\infty}}$. Then $\norm{1}_{s(\cdot)} = [p]_{\mathcal{N}}$. Let $E\coloneqq \{x\in \mathbb{R}^{n}:\,\,p'(x)\geq p'_{\infty}\}\,\,\,\mbox{and}\,\,\,F\coloneqq\mathbb{R}^{n}\setminus E$. Then by the generalized H\"{o}lder's inequality we have
  \begin{align}
    \label{eq:aux2}
    \begin{alignedat}{3}
      \|g_{1}\indicator_{E}\|_{p'_{\infty}}&\leq 2\|\indicator_{E}\|_{s(\cdot)}\|g_{1}\|_{p'(\cdot)}&&\leq \|1\|_{s(\cdot)} &&= [p]_{\mathcal{N}}.
    \end{alignedat}
  \end{align}
  Using $\abs{g_1} \leq 1$, $p_\infty' \leq p'(x) < \infty$ on $F$ and Lemma~\ref{lem:embedding} we have
  \begin{align}
    \label{eq:aux4}
    \|g_{1}\indicator_{F}\|_{p'_{\infty}}\leq 2\, \max \set{\|g_{1}\indicator_F\|_{p'(\cdot)}, \norm{g_1 \indicator_F}_\infty} \leq 2\, \max \set{\|g\|_{p'(\cdot)}, 1} \leq 2.
  \end{align}
  Consequently, we have
  $$
  \|g_{1}\|_{p'_{\infty}}\leq \|g_{1}\indicator_{E}\|_{p'_{\infty}} +\|g_{1}\indicator_{F}\|_{p'_{\infty}}\leq [p]_{\mathcal{N}} +2.
  $$
  This and~\eqref{eq:aux3} for $P=p_\infty'$ implies
  \begin{align*}
    \|T^{(2)}_{l}g_{1}\|_{p'_{\infty}}\leq 2^{\frac{n}{p'_\infty}}3^{\frac{-nl}{p'_\infty}}([p]_{\mathcal{N}} +2) \leq 2^{n}([p]_{\mathcal{N}} +2)3^{\frac{-nl}{(p')^+}} .
  \end{align*}

  Let us now consider the case $P= (p')^+$. As in~\eqref{eq:aux4} we estimate with Lemma~\ref{lem:embedding}
  \begin{align}
    \label{eq:aux5}
    \|g_{1}\indicator_{F}\|_{(p')^+}\leq 2\, \max \set{\|g_{1}\indicator_F\|_{p'(\cdot)}, \norm{g_1 \indicator_F}_\infty} \leq 2\, \max \set{\|g\|_{p'(\cdot)}, 1} \leq 2.
  \end{align}
  From~\eqref{eq:aux3}
  we obtain
  \begin{align*}
    \|T^{(2)}_{l}g_{1}\|_{(p')^{+}}\leq 2 \cdot 2^{\frac{n}{(p')^+}} 3^{- \frac{nl}{(p')^+}} \leq 2^{n+1} 3^{- \frac{nl}{(p')^+}}.
  \end{align*}
  Considering above estimates we have 
  \begin{align*}
    \max{\left\lbrace \|T^{(2)}_{l}g_{1}\|_{p'_{\infty}},  \|T^{(2)}_{l}g_{1}\|_{(p')^+}\right\rbrace} \leq 2^n \big([p]_{\mathcal{N}}+2\big) \big(3^{-\frac{n}{(p')^+}}\big)^l.
  \end{align*}
  It follows by Lemma~\ref{lem:Nekvinda-minimax} applied to $p'$ that
  \begin{align}\label{eq:max}
    \begin{aligned}
      \norm{T^{(2)}_{l}g_{1}}_{p'(\cdot)} &\leq 4 [p]_{\mathcal{N}} \max{\left\lbrace \|T^{(2)}_{l}g_{1}\|_{p'_{\infty}}, \|T^{(2)}_{l}g_{1}\|_{(p')^+}\right\rbrace}
      \\
      &\leq 4 [p]_{\mathcal{N}} \cdot 2^{n} \big([p]_{\mathcal{N}}+2\big) \big(3^{-\frac{n}{(p')^+}}\big)^l.
    \end{aligned}
  \end{align}

  We also need to estimate $\|T^{(2)}_{l}g_{2}\|_{p'(\cdot)}$.
  We have
  \begin{align*}
    \int_{\mathbb{R}^{n}}|T^{(2)}_{l}g_{2}(x)|^{\max \set{p'(x),p'_\infty}}dx
    &=\sum_{(j,k)\in \mathcal{I}_{2}}\int_{Q_{j}^{k}\cap D_{k+l}}(\alpha_{j,k}(g_{2}))^{\max \set{p'(x),p'_\infty}}dx
    \\
    &\leq\sum_{(j,k)\in \mathcal{I}_{2}}\int_{Q_{j}^{k}\cap D_{k+l}}\alpha_{j,k}(g_{2})dx
    \\
    &=\sum_{(j,k)\in \mathcal{I}_{2}}\frac{|Q_{j}^{k}\cap D_{k+l}|}{|Q_{j}^{k}|}\int_{E_{j,k}}|g_{2}(x)|dx
    \\
    &\leq\sum_{(j,k)\in \mathcal{I}_{2}}\frac{|Q_{j}^{k}\cap D_{k+l}|}{|Q_{j}^{k}|}\int_{E_{j,k}}|g_{2}(x)|^{p'(x)}dx
    \\
    &\leq 2^n3^{-nl}\int_{\mathbb{R}^{n}}|g|^{p'(x)}dx
    \\
    &\leq2^n3^{-nl}
  \end{align*}
  From this and~\eqref{eq:trick}, we get the following estimate
  \begin{align}
    \|T^{(2)}_{l}(g_{2})\|_{\max \set{p'(\cdot),p'_\infty}} \leq 2^{n} 3^{-\frac{nl}{(p')^+}}.
  \end{align} 
  Using Lemma \ref{lem:Nekvinda-minimax} for the exponent $p'(\cdot)$ we obtain that 
  \begin{align}\label{eq:constant2}
    \|T^{(2)}_{l}(g_{2})\|_{p'(\cdot)}\leq 2[p]_{\mathcal{N}} 2^n \big(3^{-\frac{n}{(p')^+}}\big)^l.
  \end{align}
  Using \eqref{eq:max} and \eqref{eq:constant2} we have
  \begin{align}\label{eq:aver2}
    \|T_l^{(2)}g\|_{p'(\cdot)}\leq
    \big( 4 [p]_{\mathcal{N}} \cdot 2^{n} ([p]_{\mathcal{N}}+2) +2^n\big)
    \big(3^{\frac{n}{(p')^+}}\big)^{-l}
  \end{align}
  The estimates in~\eqref{eq:Aver1} and~\eqref{eq:aver2} and $T g = \sum_l T^{(1)}_l g + \sum_l T^{(2)} g$ allow us to estimate $\norm{T g}_{\pdx}$ by a geometric sum. We obtain \eqref{eq:avergeneral}, which concludes the proof of the theorem.  Recall that the constants depend solely on $(p')^+$, $[p]_{\mathcal{A}}$ and $[p]_{\mathcal{N}}$.
\end{proof}

\printbibliography	
\end{document}